\documentclass{amsart}
\usepackage{amsthm,amssymb,a4wide,verbatim}

\usepackage{bm}

\newcommand{\Ctwo}{\ensuremath \mathsf{Co}_2}
\newcommand{\Mathieu}{\ensuremath 2^{11}\colon\mathsf{M}_{24}}
\newcommand{\Suz}{\ensuremath 3.\mathsf{Suz}.2}
\newcommand{\upa}[1]{\ensuremath \smash \uparrow^{#1}}
\newcommand{\BS}{\ensuremath \mathbf{BS}}
\newcommand{\PS}{\ensuremath \mathbf{PS}}

\theoremstyle{plain}\newtheorem{Theorem}{Theorem}[section]
\theoremstyle{plain}\newtheorem{Conjecture}[Theorem]{Conjecture}
\theoremstyle{plain}\newtheorem{Corollary}[Theorem]{Corollary}
\theoremstyle{plain}\newtheorem{Lemma}[Theorem]{Lemma}
\theoremstyle{plain}
\theoremstyle{definition}\newtheorem{Definition}[Theorem]{Definition}
\theoremstyle{definition}
\theoremstyle{definition}
\theoremstyle{definition}
\theoremstyle{definition}
\theoremstyle{definition}\newtheorem{Notation}[Theorem]{Notation}
\theoremstyle{definition}
\theoremstyle{definition}
\theoremstyle{definition}\newtheorem{Strategy}[Theorem]{Strategy}
\theoremstyle{definition}
\theoremstyle{definition}
\theoremstyle{definition}
\theoremstyle{definition}\newtheorem{Notation/Definition}[Theorem]{Notation/Definition}
\theoremstyle{definition}\newtheorem{noth}[Theorem]{}

\def\Z{{\mathbb Z}}
\def\N{{\mathbb N}}

\def\dim{\mathrm{dim}}           
           
\def\Ext{\mathrm{Ext}}           

\def\Hom{\mathrm{Hom}}

\def\Irr{\mathrm{Irr}}           
\def\IBr{\mathrm{IBr}}

\newcommand{\gen}[1]{\langle #1 \rangle}

\begin{document}

\begin{center}
   {\Large\bf  Brou{\'e}'s abelian defect group conjecture and
   $3$-decomposition
\\ \vspace*{0.3em} numbers 
   of the sporadic simple Conway group {\sf Co}$_1$}
\end{center}

\bigskip
\bigskip

\begin{center} {\large
        {\bf Shigeo Koshitani}$^{\text a,*}$, 
        {\bf J{\"u}rgen M{\"u}ller}$^{\text b}$,
        {\bf Felix Noeske}$^{\text c}$  
} \end{center}

\bigskip

\begin{center} {\it
${^{\mathrm{a}}}$Department of Mathematics, Graduate School of Science, \\
Chiba University, Chiba, 263-8522, Japan \\
${^{\mathrm{b, \, c}}}$Lehrstuhl D f{\"u}r Mathematik,
RWTH Aachen University, 52062 Aachen, Germany}
\end{center}

\footnote{
$^*$ Corresponding author. \\
\indent {\it E-mail addresses:} 
koshitan@math.s.chiba-u.ac.jp (S.~Koshitani), \\
juergen.mueller@math.rwth-aachen.de (J.~M{\"u}ller),
Felix.Noeske@math.rwth-aachen.de (F.~Noeske).  
}
 
\hrule 

\medskip

{\small\noindent{\bf Abstract}

\medskip\noindent
In the representation theory of finite groups, 
Brou{\'e}'s abelian defect group conjecture says that for any prime $p$, 
if a $p$-block $A$ of a finite group $G$ has an abelian defect group $P$, 
then $A$ and its Brauer corresponding block $B$ of the normaliser $N_G(P)$ 
of $P$ in $G$ are derived equivalent. We prove that Brou{\'e}'s conjecture,
and even Rickard's splendid equivalence conjecture, are true for the 
unique $3$-block $A$ of defect $2$ of the sporadic simple Conway group 
{\sf Co}$_1$, implying that both conjectures hold for all $3$-blocks
of {\sf Co}$_1$. To do so, we determine the $3$-decomposition numbers
of $A$, and we actually show that $A$ is Puig equivalent to the principal
$3$-block of the symmetric group $\mathfrak S_6$ of degree $6$.

\medskip\noindent
{\it Keywords:} Brou{\'e}'s conjecture; abelian defect group;
splendid derived equivalence; sporadic simple Conway group;
$3$-decomposition numbers.}

\medskip\noindent
{\it MSC:} 20C, 20D

\medskip

\hrule

\bigskip

\section{Introduction and notation}\label{intro}

\noindent
In the representation theory of finite groups, one of the
most important and interesting problems
is to give an affirmative answer to a conjecture
which was introduced by Brou{\'e} around 1988
\cite{Broue1990}.
He actually conjectures the following, where 
the various notions of equivalences used are recalled more
precisely in {\bf\ref{NotationEquivalences}}:

\begin{Conjecture} 
[Brou\'e's Abelian Defect Group Conjecture \cite{Broue1990}]
\label{ADGC}
Let $(\mathcal K, \mathcal O, k)$ be a
splitting $p$-modular system, where $p$ is a prime,
for all subgroups of a finite group $G$. 
Assume that $A$ is a block algebra of
$\mathcal OG$ with a defect group $P$ and that $B$ is a
block algebra of $\mathcal ON_G(P)$ such that $B$ is the
Brauer correspondent of $A$, where $N_G(P)$ is the normaliser of
$P$ in $G$. Then $A$ and $B$ should be derived equivalent
provided $P$ is abelian.
\end{Conjecture}

\noindent
In fact, a stronger conclusion is expected:

\begin{Conjecture}
[Rickard's Splendid Equivalence Conjecture \cite{Rickard1996,Rickard1998}]
\label{RickardConjecture}
Keeping the notation of {\bf\ref{ADGC}}, 
and still supposing $P$ to be abelian,
then there should be a splendid derived equivalence between
the block algebras $A$ of $\mathcal OG$ and $B$ of $\mathcal ON_G(P)$.
\end{Conjecture}

\noindent
Conjectures {\bf\ref{ADGC}} and {\bf\ref{RickardConjecture}}
have been verified for several cases,
albeit the general conjecture is widely open still.
For an overview see \cite{ChuangRickard}; in particular,
by \cite{Linckelmann1991,Rickard1989,Rouquier1995,Rouquier1998} 
{\bf\ref{ADGC}} and {\bf\ref{RickardConjecture}} are
proved for blocks with cyclic defect groups in arbitrary 
characteristic.

Moreover, it is shown in \cite[(0.2)Theorem]{KoshitaniKunugi2002}
that conjectures {\bf\ref{ADGC}} and {\bf\ref{RickardConjecture}} hold for 
the principal block algebra of an arbitrary finite group 
when the defect group is elementary abelian of order $9$.
In view of the strategy used in \cite{KoshitaniKunugi2002},
and of a possible future theory reducing conjectures {\bf\ref{ADGC}} and 
{\bf\ref{RickardConjecture}} to the quasi-simple groups,
it seems worthwhile to proceed with this class of groups,
as far as non-principal $3$-blocks with elementary abelian 
defect group of order $9$ are concerned. Indeed, for these cases 
there are partial results already known, see
\cite{KessarLinckelmann2010,KoshitaniKunugiWaki2002,
      KoshitaniKunugiWaki2004,KoshitaniKunugiWaki2008,
      KoshitaniMueller2010,Kunugi,
      KoshitaniMuellerNoeske2013a,KoshitaniMuellerNoeske2013b,
      MuellerSchaps2008}
for instance; all the non-principal
$3$-blocks with elementary abelian defect group of order $9$
for the sporadic simple groups and their covering groups are
given in \cite{Noeske2008}.

\bigskip
\noindent
The present paper is another step in that programme, 
our main result being the following:

\begin{Theorem}\label{MainTheorem}
Let $G$ be the sporadic simple Conway group {\sf Co}$_1$, and 
let $(\mathcal K, \mathcal O, k)$ be a splitting $3$-modular
system for all subgroups of $G$.
Suppose that $A$ is the unique block algebra of $\mathcal OG$
with elementary abelian defect group $P = C_3 \times C_3$ of order $9$,
and that $B$ is the block algebra
of $\mathcal O N_G(P)$ such that $B$ is the Brauer correspondent
of $A$. Then, $A$ and $B$ are splendidly derived
equivalent, hence Conjectures 
{\bf\ref{ADGC}} and {\bf\ref{RickardConjecture}} of
Brou{\'e} and Rickard hold for the block algebra $A$.
\end{Theorem}

\noindent
As an immediate corollary we get:

\begin{Corollary}\label{ADGCforCo1forAllBlocks}
Brou\'e's abelian defect group conjecture {\bf\ref{ADGC}}, and 
Rickard's splendid equivalence conjecture {\bf\ref{RickardConjecture}} 
are true for the prime $p=3$ and 
for all block algebras of $\mathcal O G$.
\end{Corollary}

\noindent
Actually, we prove {\bf\ref{MainTheorem}} by proving the following:

\begin{Theorem}\label{MoritaEquivalentToS6}
The block algebra $A$ of $G$ 
and the principal $3$-block algebra $A'$ of the symmetric group
$\mathfrak S_6$ of degree $6$ are Puig equivalent.
\end{Theorem}

\noindent
Hence, recalling the notion from 
\cite[Remark 1.6(c)]{KoshitaniMuellerNoeske2013a},
this shows that $A$ is pseudo-principal.
Moreover, as an immediate corollary we get:

\begin{Corollary}\label{OkuyamaWaki}
The block algebra $A$ of $G$ 
and the principal $3$-block algebra of the symplectic group
${\mathrm{Sp}}_4(q)$, for any prime power $q$ such that
$q \equiv 2 \ {\mathrm or } \ 5 \ {\mathrm{(mod \ 9)}}$,
are Puig equivalent. 
\end{Corollary}

\noindent 
Typically, the starting point to proving theorems like those above
is the decomposition matrix of the block in question. Alone,
for the block algebra $A$ these have not been known before.
So we first set out to tackle this problem, and arrived at 
a fairly good approximation to the decomposition matrix of $A$,
but we have not been able to determine it completely. Still,
the approximation was good enough to compare the result
with the data in \cite{KoshitaniKunugi2002} to arrive
at the sensible guess that $A$ might be closely related to
the principal block algebra $A'$ of the symmetric group
$\mathfrak S_6$, whose decomposition matrix is well-known of course. 
Now, using the partial results on decomposition numbers,
it was possible to activate the sophisticated block-theoretic
machinery needed to actually compare the block algebras $A$ and $A'$, 
which in turn finally paved the way to complete the decomposition 
matrix of $A$.
Here it is:

\begin{Theorem}\label{3DecompositionMatrixOfA}
The $3$-decomposition matrix of $A$ and of $A'$ is the following,
where we denote ordinary characters by their degrees, but also give
the ATLAS notation for the characters in $A$:
$$ 
\begin{array}{ccc}
\begin{array}{l|r|l||ccccc}
&A & A' &S_1 &S_2 &S_3 &S_4 &S_5 \\
\hline
\chi_{29}&  2\,816\,856 & 1      & 1& .& .& .& . \\
\chi_{38}& 16\,347\,825 & 5'     & 1& 1& .& .& . \\
\chi_{51}& 44\,013\,375 & 1^-    & .& .& 1& .& . \\
\chi_{55}& 57\,544\,344 & 5^-    & .& 1& 1& .& . \\
\chi_{62}& 91\,547\,820 & 10^-   & .& 1& .& 1& . \\
\chi_{80}&251\,756\,505 & 5      & 1& .& .& .& 1 \\
\chi_{85}&292\,953\,024 & {5'}^- & .& .& 1& .& 1 \\
\chi_{89}&326\,956\,500 & 10     & .& .& .& 1& 1 \\
\chi_{91}&387\,317\,700 & 16     & 1& 1& 1& 1& 1 \\
\end{array} 
&
&
\begin{array}{l}
 \dim S_1 = 2\,816\,856\\
 \dim S_2 = 13\,530\,969\\
 \dim S_3 = 44\,013\,375\\
 \dim S_4 = 78\,016\,851\\
 \dim S_5 = 248\,939\,649
\end{array}
\end{array}
$$
\end{Theorem}

\smallskip\noindent
We just remark that it is an ongoing project, see \cite{ModularAtlasProject},
to determine all the decomposition numbers of all the finite simple
and closely related groups occurring in the {\sf Atlas} \cite{Atlas}.
The above result is also a contribution to this project.
\qed

\bigskip\noindent
Following the general recipe indicated above, the present paper
is organised as follows:
In \S\ref{Felix} we set out to find an approximation
to the decomposition matrix of $A$, by using techniques
from computational modular character theory. The result,
where only a parameter $\alpha\in\{0,1,2\}$ is yet undetermined,
is given in {\bf\ref{DecompositionMatrixOfA}},
In \S\ref{A'} we collect the necessary facts for the 
symmetric group $\mathfrak S_6$ and its principal block $A'$.
In particular, we determine the trivial-source modules in $A'$,
which play a particularly important role in the sequel.
In \S\ref{betweenAandB} we prove that the block $A$ and
its Brauer correspondent $B$ in $N_G(P)$ are splendidly
stably equivalent of Morita type. This immediately implies
that $A$ and $A'$ also are splendidly
stably equivalent of Morita type.
In \S\ref{image} we finally show that the stable equivalence
between $A$ and $A'$ respects simple modules,
by a thorough consideration of trivial-source modules in $A$.
This implies that $A$ and $A'$ actually are Puig equivalent,
from which the other assertions follow immediately.

\smallskip\noindent
In order to facilitate the necessary computations,
we make use of the computer algebra system {\sf GAP} \cite{GAP},
to deal with finite groups, in particular
permutation and matrix groups, and with ordinary and Brauer
characters of finite groups. In particular, we make use of the
character table library \cite{CTblLib}, which provides electronic 
access to the data collected in the {\sf Atlas} \cite{Atlas} and in the 
{\sf ModularAtlas} \cite{ModularAtlas,ModularAtlasProject},
of the interface \cite{AtlasRep} to the database \cite{ModAtlasRep},
and of the package \cite{MaasNoeske} providing the necessary tools
from computational modular character theory.

Moreover, we use the  computer algebra system {\sf MeatAxe} \cite{MA},
and its extensions \cite{LuxMueRin, LuxSzokeII, LuxSzoke, LuxWie} 
to deal with matrix representations over finite fields.
Here, we use `small' finite fields, but we always make sure, 
silently, that these are chosen such that the computational results
thus obtained remain valid without change after scalar extension to
the fixed field of positive characteristic which is `large enough'
in the sense of {\bf\ref{NotationEquivalences}} below.

\begin{Notation/Definition}\label{NotationEquivalences}
Throughout this paper, we use the standard 
notation and terminology as is used in the 
{\sf Atlas} \cite{Atlas} and textbooks like
\cite{NagaoTsushima,Thevenaz}. We recall a few for convenience:

\smallskip
(i) If $A$ and $B$ are finite dimensional $k$-algebras, where $k$ is a field,
we denote by $\mathrm{mod}{\text -}A$,
$A{\text -}\mathrm{mod}$ and
$A{\text -}\mathrm{mod}{\text -}B$
the categories of finitely generated right $A$-modules,
left $A$-modules and $(A,B)$-bimodules, respectively.
A module always refers to a finitely generated right module, 
unless stated otherwise. Let $M$ be an $A$-module.
We let $M^\vee =\Hom_A(M_A, A_A)$ be the $A$-dual of $M$, 
so that $M^\vee$ becomes a left $A$-module. 
We denote by ${\mathrm{soc}}(M)$ and ${\mathrm{rad}}(M)$
the socle and the radical of $M$, respectively.
For non-isomorphic simple $A$-modules $S_1, \ldots, S_n$, 
and positive integers $a_1, \ldots, a_n$,
we write that 
"\emph{$M = a_1 \cdot S_1 + \cdots + a_n \cdot S_n$,
as composition factors}"
when the composition factors are
$a_1$ times $S_1$,$\ldots$ , $a_n$ times $S_n$.
For another $A$-module $N$, we write $M|N$ if 
$M$ is isomorphic to a direct summand of $N$ as an $A$-module;
and we set $[M, N]^A = \dim_k({\mathrm{Hom}}_A(M,N))$.

\smallskip
(ii)
By $G$ we always denote a finite group, and we fix a
prime number $p$. Assume that $(\mathcal K, \mathcal O, k)$ is a
splitting $p$-modular system for all subgroups of $G$, that is, 
$\mathcal O$ is a complete discrete valuation ring of
rank one such that its quotient field $\mathcal K$ is
of characteristic zero, and its residue field
$k=\mathcal O/\mathrm{rad}(\mathcal O)$ is of
characteristic $p$, and that $\mathcal K$ and $k$ are
splitting fields for all subgroups of $G$.

We denote by $\mathrm{Irr}(G)$ and $\mathrm{IBr}(G)$ the sets of
all irreducible ordinary and Brauer characters of $G$, respectively.
If $A$ is a block algebra of $\mathcal OG$,
then we write $\mathrm{Irr}(A)$ and $\mathrm{IBr}(A)$ for
the sets of all characters in $\mathrm{Irr}(G)$ and $\mathrm{IBr}(G)$
which belong to $A$, respectively.

We say that a $kG$-module $M$ is a {\sl trivial-source} module
if $M$ is indecomposable and $M$ has a trivial source,
see \cite[II Definition 12.1]{Landrock}; note that the definition
here is slightly different from \cite[\S 27 p.218]{Thevenaz}
where indecomposability is not assumed. Thus, in particular,
projective indecomposable modules are trivial-source modules. 
We recall the following facts,
see \cite[II Theorem 12.4, I Proposition 14.8]{Landrock}:
If $M$ is a trivial-source module, then $M$ lifts uniquely (up to isomorphism) 
to a trivial source $\mathcal O G$-lattice $\widehat M$,
in particular this associates an ordinary character $\chi_M$ to $M$.
Moreover, if $N$ is another trivial-source $kG$-module,
then $[M,N]^G = \gen{\chi_M,\chi_N}_G$, where $\gen{-,-}_G$ 
denotes the usual scalar product on ordinary characters.

\smallskip
(iii)
Let $G'$ be another finite group, and let $V$ be an 
$(\mathcal OG, \mathcal OG')$-bimodule. 
Then we can regard $V$ as a right
$\mathcal O[G \times G']$-module
via $v{\cdot}(g,g') = g^{-1}vg'$ for $v \in V$ and $g, g' \in G$.
Let $A$ and $A'$ be block algebras of $\mathcal OG$ and $\mathcal OG'$, 
respectively, such that $A$ and $A'$ have a defect group $P$
in common. 
Then we say that $A$ and $A'$ are {\sl Puig equivalent}
if there is a Morita equivalence between $A$ and $A'$
which is induced by an 
$(A,A')$-bimodule $\mathfrak M$ 
such that, as a right $\mathcal O[G \times G']$-module, 
$\mathfrak M$ is a trivial-source module and $\Delta P$-projective. 
This is equivalent to a condition that $A$ and $A'$ have 
source algebras which are isomorphic as interior $P$-algebras, 
see \cite[Remark 7.5]{Puig1999} and \cite[Theorem 4.1]{Linckelmann2001}.
We say that $A$ and $A'$ are {\sl stably equivalent of Morita type} 
if there exists an 
$(A, A')$-bimodule $\mathfrak M$ such that both
${_A}\mathfrak M$ and $\mathfrak M_{A'}$ are projective and that 
$_A(\mathfrak M \otimes_{A'} \mathfrak M^\vee)_A 
  \cong {_A}{A}{_A} \oplus
 ({\mathrm{proj}} \ (A,A){\text{-}}{\mathrm{bimod}})$
and
$_{A'} (\mathfrak M^\vee \otimes_A \mathfrak M)_{A'} 
\cong {_{A'}}{A'}{_{A'}} \oplus
 ({\mathrm{proj}} \ (A' ,A'){\text{-}}{\mathrm{bimod}})$.
We say that $A$ and $A'$ are 
{\sl splendidly stably equivalent of Morita type}
if the stable equivalence of Morita type is induced by
an indecomposable $(A,A')$-bimodule $\mathfrak M$ which is a trivial-source
$\mathcal O[G \times G']$-module and is $\Delta P$-projective,
see \cite[Theorem~3.1]{Linckelmann2001}.

We say that $A$ and $A'$ are 
{\sl derived equivalent} if
${\mathrm{D}}^b({\mathrm{mod}}{\text{-}}A)$ and
${\mathrm{D}}^b({\mathrm{mod}}{\text{-}}A')$ are
equivalent as triangulated categories,
where 
${\mathrm{D}}^b({\mathrm{mod}}{\text{-}}A)$
is the bounded derived category of ${\mathrm{mod}}{\text{-}}A$.
In that case, there even is a {\sl Rickard complex}
$M^{\bullet} \in
{\mathrm{C}}^b (A{\text{-}}{\mathrm{mod}}{\text{-}}A')$,
where the latter 
is the category of bounded complexes
of finitely generated $(A,A')$-bimodules,
all of whose terms are projective both as 
left $A$-modules and as right $A'$-modules, such that 
$M^{\bullet}\otimes_{A'}(M^{\bullet})^{\vee}\cong A$
in $K^b(A{\text{-}}{\mathrm{mod}}{\text{-}}A)$
and
$(M^{\bullet})^{\vee}\otimes_A M^{\bullet}\cong A'$
in $K^b(A'{\text{-}}{\mathrm{mod}}{\text{-}}A')$,
where $K^b(A{\text{-}}{\mathrm{mod}}{\text{-}}A)$ is the homotopy
category associated with 
${\mathrm{C}}^b (A{\text{-}}{\mathrm{mod}}{\text{-}}A)$.
We say that $A$ and $A'$ are 
{\sl splendidly derived equivalent} 
if $K^b({\mathrm{mod}}{\text{-}}A)$ and 
$K^b({\mathrm{mod}}{\text{-}}A')$ are equivalent
via a Rickard complex 
$M^{\bullet}\in {\mathrm{C}}^b (A{\text{-}}{\mathrm{mod}}{\text{-}}A')$
as above,
such that additionally each of its terms is a 
direct sum of $\Delta P$-projective 
trivial-source modules as an
$\mathcal O[G \times G']$-module;
see \cite{Linckelmann1998,Linckelmann2001}.
Note that a Morita equivalence entails a derived equivalence,
and that a Puig equivalence entails a splendid derived equivalence.
\end{Notation/Definition}

\section{Decomposition numbers of {\sf Co}$_1$}\label{Felix}

In this section we use methods from computational modular character theory
to find an approximation to the decomposition matrix of the block in
question. To do so, we make heavy use of the data on decomposition 
numbers known for various maximal subgroups of $G$,
as is contained in \cite{ModularAtlas, ModularAtlasProject}.
We stress that, to find the approximate decomposition matrix
exhibited below in the first place, we make use of the full machinery
laid out in \cite{MOCbook}, in combination with substantial computation,
using {\sf GAP} \cite{GAP} and the package \cite{MaasNoeske},
while the proof presented here is subsequently derived from the
computer-based one by straightening it out by hand. 

\smallskip\noindent
For convenience we recall a few of the basic notions used in
this approach, while for full details we
refer the reader to \cite{MOCbook}:

\begin{Definition}\label{DefinitionCo1}
We keep the notation in {\bf\ref{NotationEquivalences}}(ii).
Let $\Z\mathrm{IBr}(A)$ be the lattice of generalised
Brauer characters belonging to $A$, and
$\N_0\mathrm{IBr}(A)\subset\Z\mathrm{IBr}(A)$ 
be the non-negative cone of genuine Brauer characters 
spanned by $\mathrm{IBr}(A)$.
Then a $\Z$-basis $\BS$ of $\Z\mathrm{IBr}(A)$ which is contained
in $\N_0\mathrm{IBr}(A)$ is called a {\it basic set} of Brauer characters;
in particular, $\mathrm{IBr}(A)$ is such a basic set.

Moreover, any projective $A$-module has an ordinary character associated
with it. Thus running through the projective indecomposable $A$-modules,
this gives rise to the set $\mathrm{IPr}(A)\subset\N_0\mathrm{Irr}(A)$,
spanning the lattice $\Z\mathrm{IPr}(A)$ of generalised projective
characters belonging to $A$.
Then, similarly, a $\Z$-basis $\PS$ of $\Z\mathrm{IPr}(A)$ 
which is contained in $\N_0\mathrm{IPr}(A)$ is called a {\it basic set} 
of projective characters;
in particular, $\mathrm{IPr}(A)$ is such a basic set.

Now let $\widehat{\mathrm{Irr}}(A)$ be the set of restrictions of the
characters in $\mathrm{Irr}(A)$ to the $p$-regular conjugacy classes 
of $G$, thus we have $\widehat{\mathrm{Irr}}(A)\subset\N_0\mathrm{IBr}(A)$.
Similarly we obtain $\widehat{\mathrm{IPr}}(A)$, and thus we get
$\Z\widehat{\mathrm{IPr}}(A)\subseteq\Z\mathrm{IBr}(A)$ such that
$\widehat{\mathrm{IPr}}(A)\subset
 \gen{\widehat{\mathrm{Irr}}(A)}_{\N_0}\subseteq
 \N_0\mathrm{IBr}(A)$.
The restriction $\gen{-,-}'_G$ of the character-theoretic 
scalar product to the $p$-regular classes
induces a non-degenerate pairing
between $\Z\mathrm{IBr}(A)$ and $\Z\widehat{\mathrm{IPr}}(A)$, 
such that $\mathrm{IBr}(A)$ and $\widehat{\mathrm{IPr}}(A)$ 
are a pair of mutually dual $\Z$-bases.

More generally, given basic sets $\BS$ and $\PS$ as above,
the associated matrix of mutual scalar products is
unimodular and has non-negative entries. Now the general strategy
is to use character theory to find sequences of basic sets of
Brauer characters and of projective characters, which better and
better approximate the sets $\mathrm{IBr}(A)$ and $\mathrm{IPr}(A)$,
respectively, the guiding principle being the above-mentioned
positivity properties.
\end{Definition}

\begin{Notation}\label{NotationCo1}
>From now on we use the following notation.
Let $G := {\sf Co}_1$ be the first sporadic simple Conway group, and hence
$|G| = 2^{21}{\cdot}3^9{\cdot}5^4{\cdot}7^2{\cdot}11{\cdot}13{\cdot}23$,
see \cite[p.180]{Atlas}.
Assume $p = 3$, and let $(\mathcal K, \mathcal O, k)$ 
be a splitting $3$-modular system for all finite groups we shall treat with.
It follows from the character table of $G$,
see \cite[p.180]{Atlas},
that there is a unique block algebra $A$ of $kG$ 
with elementary abelian defect group $P$ of order $9$; 
see also \cite{Noeske2008}. 
(We will specify $P$ as a subgroup of $G$ more precisely 
in {\bf\ref{FindH}} and {\bf\ref{NotationH}} below.)

Moreover, it turns out that $A$ contains $k(A)=9$ irreducible ordinary
characters and $l(A)=5$ irreducible Brauer characters. 
According to the ordering in \cite[p.180]{Atlas},
the ordinary characters in $A$ are
$\{ \chi_{29}, \chi_{38}, \chi_{51}, \chi_{55}, \chi_{62},
\chi_{80}, \chi_{85}, \chi_{89}, \chi_{91} \} $,
for convenience their degrees are indicated in 
{\bf\ref{3DecompositionMatrixOfA}}.
\end{Notation}

\begin{noth}\label{MOC1}
We begin by choosing basic sets of Brauer characters and projective
characters. 

Let $\widehat{\chi}_i$ denote the $i$-th irreducible character of $G$,
with respect to the ordering given in \cite[p.180]{Atlas},
restricted to its $3$-regular conjugacy classes. 
Moreover, for a subgroup $H\le G$ let $\psi(H)_i$ denote the $i$-th
irreducible Brauer character of $H$ with respect to the ordering
given in the databases mentioned, and let $\Psi(H)_i$ denote the 
ordinary character of the corresponding projective indecomposable module.
We remark that the irreducible Brauer characters and the projective
indecomposable characters of the $2$-local maximal subgroup $\Mathieu$
are easily computed using Fong-Reynolds correspondence, 
but we only need those of the Mathieu group $\mathsf{M}_{24}$, inflated
along the natural map $\Mathieu\rightarrow\mathsf{M}_{24}$.

Then our initial basic set $\BS=\{\varphi_1,\ldots,\varphi_5\}$ 
of Brauer characters constitutes of the $A$-parts of the Brauer characters 
\[ \widehat{\chi}_{29},\: \psi(\Suz)_7\upa{G},\: \widehat{\chi}_{51},\: 
\psi(\Suz)_{11}\upa{G},\: \widehat{\chi}_{80},\]
and as an initial basic set $\PS=\{\Phi_1, \ldots, \Phi_5\}$ 
of projective characters we choose the $A$-parts of
\[ \Psi(\Ctwo)_{18}\upa{G},\: \Psi(\Ctwo)_{22}\upa{G},\:
\Psi(\Ctwo)_{29}\upa{G},\: \Psi(\Mathieu)_{12}\upa{G},\:
\]
These indeed are suitable basic sets, as the matrix 
$(\gen{\varphi_i, \Phi_j }'_G)_{1\le i,j\le 5}$ 
shows:

 \[
 \begin{array}{|r|rrrrr|}
  \hline
  \gen{-,-}'_G & \bm{\Phi_1} & \Phi_2 & \Phi_3 & \Phi_4 & \Phi_5\\
  \hline
  \bm{\varphi_1} &.& .& .& .& 1\\
  \bm{\varphi_2} &.& .& .& 1& 1\\
  \bm{\varphi_3} &.& .& 1& .& .\\
  \bm{\varphi_4} &.& 1& .& 2& 1\\
  \varphi_5      &1& 1& 3& 1& 3\\
  \hline
 \end{array}
 \]

Note that in particular this initial matrix already turns out to 
be unitriangular, which will simplify the arguments to follow.
\end{noth}

\begin{noth}\label{MOC2}
it is now immediate that $\varphi_1$ and $\varphi_3$ 
are irreducible Brauer characters, and that $\Phi_1$ 
is the character of a projective indecomposable module. 
We denote the known irreducible Brauer characters and 
projective indecomposable characters in bold face in 

Next we prove that $\varphi_2$ and $\varphi_4$ actually are
irreducible Brauer characters as well.
To do so, we consider the $A$-parts of the projective characters 
$\Psi(\Suz)_{27}\upa{G}$ and $\Psi(\Mathieu)_{2}\upa{G}$,
whose decompositions into $\PS$ are given as follows:
\[ \begin{array}{|l|rrrrr|} \hline
& \Phi_1 & \Phi_2 & \Phi_3 & \Phi_4 & \Phi_5\\
\hline
  \Psi(\Suz)_{27}\upa{G} & -9& 21& 16& -1& 1\\
  \Psi(\Mathieu)_{2}\upa{G} & -4& -1& 0& 0& 4\\
\hline \end{array} \]

Now assume that $\varphi_2$ is reducible. Then 
$\varphi_2-\varphi_1$ is a Brauer character, whose
scalar product with $\Psi(\Suz)_{27}\upa{G}$
equals $[0,0,0,1,0]\cdot [-9,21,16,-1,1]^{\top}=-1$, 
a contradiction. Hence $\varphi_2$ is an irreducible Brauer character.

Next, assume that $\varphi_4$ is reducible. Then either 
$\varphi_4-\varphi_2$ or $\varphi_4-\varphi_1$ is a Brauer character.
Taking scalar products with $\Psi(\Mathieu)_{2}\upa{G}$ yields
$[0,1,0,1,0]\cdot [-4,-1,0,0,4]^{\top}=-1$ and
$[0,1,0,2,0]\cdot [-4,-1,0,0,4]^{\top}=-1$, respectively,      
a contradiction. Thus $\varphi_4$ is an irreducible Brauer character.
\end{noth}

\begin{noth}\label{MOC3}
We now set out to improve our basic sets $\BS$ and $\PS$.

that $\varphi_4$ is not contained in $\varphi_5$, hence we conclude
that $\Phi_2' := \Phi_2 - \Phi_1$ is a projective character,
and then as such is indecomposable.

Next, using the facts that $\varphi_4$ and $\varphi_2$ are irreducible,
we conclude that $\Phi_4' := \Phi_4 - 2(\Phi_2-\Phi_1)$ and 
$\Phi_5' := \Phi_5 - (\Phi_2-\Phi_1)-(\Phi_4 - 2\Phi_2+\Phi_1)
=\Phi_5-\Phi_4+\Phi_2$
are projective characters. 

Finally, it turns out that $\widehat{\chi}_{85}$ 
decomposes into $\BS$ as follows:
\[ \begin{array}{|l|rrrrr|} \hline
& \varphi_1 & \varphi_2 & \varphi_3 & \varphi_4 & \varphi_5\\
\hline 
\widehat{\chi}_{85}& -1 & . & 1 & . & 1\\
\hline \end{array} \]
As $\varphi_3$ is irreducible it follows 
that $\varphi'_5:=\varphi_5 - \varphi_1$ is a Brauer character.

Thus, we have obtained new basic sets
$\BS':=\{\varphi_1,\ldots,\varphi_4,\varphi'_5\}$ and 
$\PS':=\{\Phi_1, \Phi_2', \Phi_3, \Phi_4', \Phi_5' \}$,
whose matrix of mutual scalar products is given as follows:

 \[
 \begin{array}{|r|rrrrr|}
  \hline
  \gen{-,-}'_G & \bm{\Phi_1} & \bm{\Phi'_2} & 
                       \Phi_3 & \Phi'_4 & \Phi'_5 \\
  \hline
  \bm{\varphi_1} &.& .& .& .& 1\\
  \bm{\varphi_2} &.& .& .& 1& .\\
  \bm{\varphi_3} &.& .& 1& .& .\\
  \bm{\varphi_4} &.& 1& .& .& .\\
  \varphi'_5     &1& .& 3& 1& 2\\
  \hline
 \end{array}
 \]
\end{noth}

\begin{noth}\label{MOC4}
With respect to $\PS'$ we now have:
\[ \begin{array}{|l|rrrrr|} \hline
& \Phi_1 & \Phi_2' & \Phi_3 & \Phi_4' & \Phi_5' \\
\hline
  \Psi(\Mathieu)_{2}\upa{G} & -9 & 3& 0& 4& 4 \\
\hline \end{array} \]
Since $\Phi_1$ is already known to be indecomposable, 
show that it is contained 
at most once  in $\Phi_4'$ and at most twice in $\Phi_5'$,
we from this conclude that
$\Phi_5'':=\Phi_5'- 2\Phi_1$ and $\Phi_4'':=\Phi_4'-\Phi_1$
are projective characters. Then, as such, they both are indecomposable.

Finally, the $A$-part of the tensor product $\Phi''_4\otimes\widehat{\chi}_2$
decomposes into the newly found projective characters as follows:
\[ \begin{array}{|l|rrrrr|} \hline
& \Phi_1 & \Phi_2' & \Phi_3 & \Phi_4'' & \Phi_5'' \\
\hline
\Phi''_4\otimes\widehat{\chi}_2 & -4 & 2 & 4 & 2 & 0 \\
\hline \end{array} \]
Hence $\Phi'_3:=\Phi_3-\Phi_1$ is a projective character.

Thus we have obtained the new basic set
$\PS'':=\{\Phi_1 , \Phi_2' , \Phi_3' , \Phi_4'' , \Phi_5'' \}$,
whose matrix of scalar products 
with $\BS'$ is given as follows:

 \[
 \begin{array}{|r|rrrrr|}
  \hline
  \gen{-,-}'_G & \bm{\Phi_1} & \bm{\Phi'_2} & 
                       \Phi'_3 & \bm{\Phi''_4} & \bm{\Phi''_5} \\
  \hline
  \bm{\varphi_1} &.& .& .& .& 1\\
  \bm{\varphi_2} &.& .& .& 1& .\\
  \bm{\varphi_3} &.& .& 1& .& .\\
  \bm{\varphi_4} &.& 1& .& .& .\\
  \varphi'_5     &1& .& 2& .& .\\
  \hline
 \end{array}
 \]

Hence, only three possible decomposition matrices remain:
one of $\Phi_3' - \alpha\cdot \Phi_1$ for $0\le \alpha \le 2$ is a
projective indecomposable character. Thus we have proved:
\end{noth}

\begin{Lemma}\label{DecompositionMatrixOfA}
The $3$-decomposition matrix of $A$ is the following:
$$ \begin{array}{l|ccccc}
& S_1 & S_2 & S_3 & S_4 & S_5 \\
\hline
\chi_{29}& 1& .& .& .& . \\
\chi_{38}& 1& 1& .& .& . \\
\chi_{51}& .& .& 1& .& . \\
\chi_{55}& .& 1& 1& .& . \\
\chi_{62}& .& 1& .& 1& . \\
\chi_{80}& 1& .& \alpha& .& 1 \\
\chi_{85}& .& .& 1+\alpha& .& 1 \\
\chi_{89}& .& .& \alpha& 1& 1 \\
\chi_{91}& 1& 1& 1+\alpha& 1& 1 \\
\end{array} $$
where $\alpha$ is a certain integer such that $\alpha \in \{0,1,2\}$.
\end{Lemma}

\begin{proof}
This follows from {\bf\ref{MOC1}}--{\bf\ref{MOC4}},
where here we reverse the order of the projective 
characters in the final basic set, and already indicate
the notation for the associated simple $A$-modules which will
be used later.
\end{proof}

\section{The group $\mathfrak S_6$}\label{A'}

\begin{Notation}\label{NotationGprime} 
Set $G' := \mathfrak S_6$, the symmetric group of degree $6$.
Then the Sylow $3$-subgroup of $G'$ are elementary abelian of order $9$.
Hence fixing a block defect group $P$ of $A$, 
see {\bf\ref{NotationCo1}}, we may identify $P$ with a Sylow $3$-subgroup
of $G'$. Doing so, we have $P = Q \times R$ with $Q \cong C_3 \cong R$, 
where, with respect to the tautological permutation representation of $G'$, 
we may assume that the non-trivial elements of $Q$ are $3$-cycles,
while those of $R$ are fixed-point free. 

Let $A'$ be the principal block algebra of $kG'$. Then we have
$\IBr(A')=\{ 1a, 1b, 4a, 4b, 6 \}$ and 
$\Irr(A')=\{ \chi_1, \chi_{1^-}, \chi_5, \chi_{5^-}, 
\chi_{5'}, \chi_{{5'}^-}, \chi_{10}, \chi_{10^-}, \chi_{16} \}$,
where we use the following notation:

As usual, we denote by $\chi_1:=1_{A'}\in\Irr(A')$ and 
$\chi_{1^-}:=1^-\in\Irr(A')$ the trivial and the sign character,
respectively. Moreover, we let $\chi_5\in\Irr(A')$ be the
non-trivial constituent of the tautological permutation 
character of $A'$, which is characterised by having positive
values on both the $3$-cycles and the transpositions,
and we let $\chi_{5'}\in\Irr(A')$ be the irreducible
character having positive values on both the fixed-point free 
elements of order $3$ and $2$;
note that $\chi_5$ and $\chi_{5'}$ are interchanged by the 
non-trivial outer automorphism of $A'$. 
Finally, we let $\chi_{10}\in\Irr(A')$ be the irreducible
character having positive value on the transpositions.  
Then we get $\chi_{5^-}:=\chi_5\otimes 1^-\in\Irr(A')$ and
$\chi_{{5'}^-}:=\chi_{5'}\otimes 1^-\in\Irr(A')$
as well as $\chi_{10^-}:=\chi_{10}\otimes 1^-\in\Irr(A')$.

As for the simple modular representations occurring, we let, 
again as usual, $1a := k_{G'}\in\IBr(A')$ and $1b := 1^-\in\IBr(A')$
be the trivial and the sign representation, respectively.
Moreover, we let $4a$ be the non-trivial modular constituent 
of the tautological permutation representation of $G'$, then we get
$4b:=4a\otimes 1^-\in\IBr(A')$.
Note that the outer automorphism of $G'$ mentioned earlier
interchanges $4a\leftrightarrow 4b$, and leaves $1a$, $1b$, and $6$ fixed.
\end{Notation}

\begin{Lemma}\label{DecAprime}
The $3$-decomposition matrix of $A'$ is given as follows:
$$ \begin{array}{l|ccccc}
 & 1a & 1b & 4a & 4b & 6 \\ 
\hline
\chi_{1}      &  1&  .&  .&  .&  . \\
\chi_{1^-}    &  .&  1&  .&  .&  . \\
\chi_{5}      &  1&  .&  1&  .&  . \\
\chi_{5^-}    &  .&  1&  .&  1&  . \\
\chi_{5'}     &  1&  .&  .&  1&  . \\
\chi_{{5'}^-} &  .&  1&  1&  .&  . \\
\chi_{10}     &  .&  .&  1&  .&  1 \\
\chi_{10^-}   &  .&  .&  .&  1&  1 \\
\chi_{16}     &  1&  1&  1&  1&  1 \\
\end{array} $$
\end{Lemma}

\begin{proof} 
This is taken from \cite[p. 4]{ModularAtlas}.
\end{proof}

\begin{Lemma}\label{TrivialSourceModulesInAprime}
Using the notation introduced in {\bf\ref{NotationGprime}},
the Loewy and socle series of the trivial-source $kG'$-modules 
in $A'$ are given as follows, where for the non-projective ones we also
include their associated ordinary characters: 
\begin{enumerate} \renewcommand{\labelenumi}{\rm{(\roman{enumi})}}
    \item
The trivial-source $kG'$-modules in $A'$ with vertex $P$ are the following:
\[
  \begin{matrix}
1a &  1b  & \boxed{\begin{matrix} 4a\\1a \ \ 1b\\4a \end{matrix}}
          & \boxed{\begin{matrix} 4b\\1a \ \ 1b\\4b \end{matrix}} 
          & \boxed{\begin{matrix}  6\\4a \ \ 4b\\6 \end{matrix}}  \\
  \updownarrow &   \updownarrow &   \updownarrow &   \updownarrow 
& \updownarrow \\
  \chi_1  & \chi_{1^-} & \chi_5+\chi_{{5'}^-} & \chi_{5^-}+\chi_{5'} &
            \chi_{10}+\chi_{10^-}.
  \end{matrix}          
\]
   \item
The trivial-source $kG'$-modules in $A'$ with vertex $Q$ are the following:
\[
 \begin{matrix}
                \boxed{\begin{matrix} 1a\\4a\\1a \end{matrix}}
              & \boxed{\begin{matrix} 1b\\4b\\1b \end{matrix}} 
              & \boxed{\begin{matrix} 4a\\1a \ \ 6\\4a \end{matrix}} 
              & \boxed{\begin{matrix} 4b\\1b \ \ 6\\4b \end{matrix}}    
\\
  \updownarrow &   \updownarrow &   \updownarrow &   \updownarrow
\\
  \chi_1+\chi_5    & \chi_{1^-}+\chi_{5^-}  & 
  \chi_5+\chi_{10} & \chi_{5^-}+\chi_{10^-} 
 \end{matrix}          
\]
   \item
The trivial-source $kG'$-modules in $A'$ with vertex $R$ are the following:
\[
 \begin{matrix}
                \boxed{\begin{matrix} 1a\\4b\\1a \end{matrix}}
              & \boxed{\begin{matrix} 1b\\4a\\1b \end{matrix}} 
              & \boxed{\begin{matrix} 4b\\1a \ \ 6\\4b \end{matrix}} 
              & \boxed{\begin{matrix} 4a\\1b \ \ 6\\4a \end{matrix}}    
\\
  \updownarrow &   \updownarrow &   \updownarrow &   \updownarrow
\\
  \chi_1+\chi_{5'} & \chi_{1^-}+\chi_{{5'}^-}  & 
  \chi_{5'}+\chi_{10^-} & \chi_{{5'}^-}+\chi_{10} 
 \end{matrix}          
\]
   \item
The projective indecomposable $kG'$-modules in $A'$ 
have the following structure:
\[
\begin{matrix}
   \boxed{
         \begin{matrix} 1a \\ 4a \ 4b \\ 1a \ 1a \ 1b \ 6 
                     \\ 4a \ 4b \\ 1a 
         \end{matrix} 
          }
&
   \boxed{
          \begin{matrix} 1b \\ 4a \ 4b \\ 1a \ 1b \ 1b \ 6 
                     \\ 4a \ 4b \\ 1b 
          \end{matrix}  
          } 
&
   \boxed{ 
           \begin{matrix} 4a \\ 1a \ 1b \ 6    \\ 4a \ 4a \ 4b 
                     \\ 1a \ 1b \ 6    \\ 4a 
            \end{matrix} 
          } 
&
   \boxed{ 
            \begin{matrix} 4b \\ 1a \ 1b \ 6    \\ 4a \ 4b \ 4b 
                     \\ 1a \ 1b \ 6 \\ 4b
             \end{matrix} 
          } 
&
   \boxed{ 
            \begin{matrix} 6 \\ 4a \ 4b  \\ 1a \ 1b \ 6 
                    \\ 4a \ 4b  \\ 6
            \end{matrix} 
          } 
\end{matrix}          
\]
\end{enumerate}
\end{Lemma}

\begin{proof} 
(i)--(iii) 
This is found by explicit computation as follows:
By Green correspondence and \cite[Chap.4 Problem 10]{NagaoTsushima},
the modules we are interested in are precisely the 
indecomposable direct summands with maximal vertex of the
permutation $kG'$-modules on the cosets of $P$, $Q$, and $R$, respectively.
Now, the relevant permutation actions of $G'$ can be computed 
using {\sf GAP} \cite{GAP}, and going over to permutation $kG'$-modules,
their indecomposable direct summands, together with their structure,
are subsequently found using the {\sf MeatAxe} \cite{MA}.
Recall that the number of trivial-source $kG'$-modules 
with vertex as prescribed above can be determined a priority,
using the facts that $N_{G'}(P)/P\cong D_8$ and
$N_{G'}(Q)/Q\cong 2\times\mathfrak S_3\cong N_{G'}(R)/R$. 
Finally, knowing the structure of the trivial-source modules
in question, the ordinary characters associated with them
are determined using {\bf\ref{DecAprime}}.

(iv)
This is contained in \cite{Waki1989}, and can also be rechecked
computationally by applying the {\sf MeatAxe} \cite{MA} to the
regular representation of $kG'$.
\end{proof}

\begin{Notation}\label{NotationHprime}
Set $H' := N_{G'}(P)$. Then we have $H'\cong P\colon D_8$,
where it follows from the description in {\bf\ref{NotationGprime}}
that $P$ is simple, viewed as an ${\mathbb F}_3D_8$-module,
which hence determines $H'$ up to isomorphism. 
But note that, by its mere definition, $H'$ comes with a fixed
embedding into $G'$.

Let $B'$ be the principal block algebra of $kH'$,
being the Brauer correspondent of $A'$ in $H'$;
note that we actually have $B'=kH'$. 
\end{Notation}

\begin{Lemma}\label{StableEquivalenceB'andA'}
Let $\mathcal M'$ be the Scott $k[G' \times H']$-module 
with vertex $\Delta P$. (Note that this is the up to isomorphism unique 
indecomposable direct summand of the $(A',B')$-bimodule 
$A'{\downarrow}^{G' \times G'}_{G' \times H'}{\cdot}1_{B'}$ 
with vertex $\Delta P$.) 
Then the pair $(\mathcal M', \mathcal {M'}^{\vee})$ induces a
splendid stable equivalence of Morita type
between the block algebras $A'$ and $B'$.
\end{Lemma}

\begin{proof}
This follows from \cite[Example 4.4]{Okuyama1997}
and \cite[Corollary 2]{Okuyama2000}.
\end{proof}

\section{Stable equivalence between $A$ and $B$}\label{betweenAandB}

\begin{Lemma}\label{FindH}
Recall the notation $G$, $P$, and $A$ as in {\bf\ref{NotationCo1}}.
Then the following holds:
\begin{enumerate} \renewcommand{\labelenumi}{\rm{(\roman{enumi})}}
    \item
We have $H:=N_G(P) = P.U_4(3).D_8$ and $C_G(P) = P.U_4(3)$ .
    \item
The non-trivial elements in $P$ fall into two conjugacy classes of $H$.
Thus $P$ has exactly two $H$-conjugacy classes of subgroups of order $3$.
Moreover, the non-trivial elements of $P$ belong to conjugacy classes 
$3A$ and $3B$ in \cite[p.184]{Atlas}.
\end{enumerate}
\end{Lemma}

\begin{proof}
We use the smallest faithful permutation representation
of $G$ on $98280$ points, available in \cite{ModAtlasRep},
and {\sf GAP} \cite{GAP} to find a defect group $P\le G$ of $A$ explicitly. 
In order to do so, using the character table of $G$, 
see \cite[p.180]{Atlas}, it turns out that conjugacy classes 
$4D$, $5B$ $7A$, and $8A$ are defect classes of $A$.
By a random search we find an element of order $42$ in $G$,
whose $6$-th power, $x$ say, hence belongs to conjugacy class $7A$.
Thus we may let $P$ be a Sylow $3$-subgroup of $C_G(x)$.

Having $P$ found explicitly,
it turns out that $H=N_G(P)$ has order $235\,146\,240$,
and that $C_G(P)$ has order $29\,393\,280=|H|/8$.
Now a consideration of the orders of the maximal subgroups of $G$, 
see \cite[p.180]{Atlas}, shows that $H$ necessarily is a maximal
subgroup of $G$, of shape $P.U_4(3).D_8$. Then it is clear
that $C_G(P)$ is of shape $P.U_4(3)$. This shows (i).
The character table of the maximal subgroup $H$ of $G$, 
available in \cite{CTblLib}, shows that $P-\{1\}\subset H$ consists 
of two rational conjugacy classes.
a consideration of character values shows, together with 
Brauer's Second Main Theorem, that the latter conjugacy classes
fuse into the conjugacy classes $3A$ and $3B$ of $G$.
\end{proof}

\begin{Notation}\label{NotationH} 
Set $H: = N_G(P)$. Then by {\bf\ref{FindH}}(ii)
we have $P = Q \times R$ with $Q \cong C_3 \cong R$, 
such that $Q$ and $R$ are not conjugate in $G$,
and we may assume that the non-trivial elements in $Q$
and $R$ belong to conjugacy classes $3A$ and $3B$, respectively.
Recall that we have already chosen an embedding of $P$ into $G'$,
see {\bf\ref{NotationGprime}}, but since the automorphism group of $P$ acts
transitively on the minimal generating sets of $P$, both choices
are consistent, justifying the reuse of earlier notation.

Let $B$ be a block algebra of $kH$
which is the Brauer correspondent of $A$.
Let $(P,e)$ be a maximal $A$-Brauer pair in $G$, namely,
$e$ is a block idempotent of $kC_G(P)$ such that 
${\mathrm{Br}}_P(1_A){\cdot}e = e$, see \cite{AlperinBroue},
\cite{BrouePuig1980} and \cite[\S 40]{Thevenaz}.
Let $i$ and $j$ respectively be source idempotents of
$A$ and $B$ with respect to $P$.
As remarked in \cite[pp.821--822]{Linckelmann2001},
we can take $i$ and $j$ such that
${\mathrm{Br}}_P(i){\cdot}e = 
 {\mathrm{Br}}_P(i) \not= 0$ 
and that
${\mathrm{Br}}_P(j){\cdot}e = 
 {\mathrm{Br}}_P(j) \not= 0$.

Set $G_P = C_G(P) = C_H(P) = H_P$.
We set $G_Q = C_G(Q)$ and $H_Q = C_H(Q)$. 
By replacing $e_Q$ and $f_Q$ (if necessary), we may assume that
$e_Q$ and $f_Q$ respectively are block idempotents of $kG_Q$ and
$kH_Q$ such that $e_Q$ and $f_Q$ are determined 
by $i$ and $j$, respectively.
Namely, 
${\mathrm Br}_Q(i){\cdot}e_Q =
 {\mathrm Br}_Q(i)$
and
${\mathrm Br}_Q(j){\cdot}f_Q =
 {\mathrm Br}_Q(j)$.
Let $A_Q = kG_Q{\cdot}e_Q$ and 
$B_Q = kH_Q{\cdot}f_Q$, so that
$e_Q = 1_{A_Q}$ and $f_Q = 1_{B_Q}$.
Similarly we define $G_R$, $H_R$, $A_R$, $B_R$, $e_R$ and $f_R$.
\end{Notation}

\begin{Lemma}\label{StructureOfB}
The following holds:
\begin{enumerate} \renewcommand{\labelenumi}{\rm{(\roman{enumi})}}
    \item
We have $H = N_G(P,e) := \{ g \in N_G(P) | g^{-1}e g = e \}$.
    \item 
$B = kHe \cong {\mathrm{Mat}}_{729}(k[P : D_8])$, where $P : D_8\cong H'$.
Hence we may write
$\IBr(B) := \{ 729a, 729b, 729c, 729d, 1458 \}$.
    \item
The blocks $B$ and $B'$ are canonically Puig equivalent.
\end{enumerate}
\end{Lemma}

\begin{proof}
(i)
The group $C_G(P)/P=U_4(3)$, being a simple group of Lie type in
characteristic $3$, has a unique irreducible character of defect zero,
namely the Steinberg character of degree $729$. Hence we conclude that
$C_G(P)$ has a unique block with defect group $P$, and thus we have 
$N_G(P,e)=N_G(P)=H$.

(ii) 
We argue similar to the lines of
\cite[Proof of Lemma 4.6(iv)]{KoshitaniMuellerNoeske2013a}:
By \cite[A.Theorem]{Kuelshammer1985}
we have $B\cong{\mathrm{Mat}}_{729}(k^{\alpha}[P : D_8])$,
for some cocycle $\alpha \in {\mathrm{Z}}^2(D_8, k^\times)$,
and where the action of $D_8\cong N_G(P)/C_G(P)$ 
in the semidirect product $P : D_8$ is determined 
by the action of $N_G(P)$ on $P$.
Now it follows from {\bf\ref{FindH}}(ii) that $P$ is simple
as an ${\mathbb F}_3D_8$-module, hence by {\bf\ref{NotationHprime}} 
we have $P : D_8\cong H'$.

Since, by \cite[Satz V.25.6]{Huppert1966} we have
$|{\mathrm{H}}^2(D_8, k^\times)| = 2$, it remains to show
that $\alpha \equiv 1$ (mod ${\mathrm{B}}^2(D_8, k^\times)$).
Indeed, using the character table of $H$,
{\sf GAP} \cite{GAP} shows that $B$ has nine irreducible
characters and five irreducible Brauer characters,
from which the last part of (ii) follows 
using \cite[p.34, Tbl.1]{Kiyota1984}.

Finally, in (iii),
the statement about source algebras follows from
\cite[Proposition 14.6]{Puig1988JAlg},
see \cite[(45.12)Theorem]{Thevenaz} and
\cite[Theorem 13]{AlperinLinckelmannRouquier}.
\end{proof}

\begin{Lemma}\label{CenQ}
The following holds:
\begin{enumerate} \renewcommand{\labelenumi}{\rm{(\roman{enumi})}}
    \item
$G_Q = Q.\mathsf{Suz}$, so that $G_Q/Q \cong \mathsf{Suz}$.
    \item
$H_Q = P.U_4(3).2'_3$, so that $H_Q/Q \cong 3_2.U_4(3).2'_3$.
    \item
$G_R = H_R = P.U_4(3).2'_2$, so that 
$G_R/R = H_R/R \cong 3_1.U_4(3).2'_2$.
\end{enumerate}
\end{Lemma}

\begin{proof}
By \cite[p.183]{Atlas} we have
$N_G(Q)=Q.\mathsf{Suz}.2$, a maximal subgroup of $G$, implying (i).

To show (ii), a consideration of the ${\mathbb F}_3D_8$-module shows
that $N_H(Q)$ has index $2$ in $H$, and is of shape $P.U_4(3).2^2$.
Thus $H_Q$ is of shape $P.U_4(3).2$, and hence $H_Q/Q$ is 
of shape $3.U_4(3).2$.
To find the precise structure, we note that $H_Q/Q$ is a subgroup
of $G_Q/Q \cong \mathsf{Suz}$, and a consideration of the 
orders of the maximal subgroups of $\mathsf{Suz}$, 
see \cite[p.131]{Atlas}, shows that $H_Q/Q$
necessarily is a maximal subgroup of $\mathsf{Suz}$,
namely the normaliser of a cyclic subgroup of order $3$ whose 
non-trivial elements belong the conjugacy class $3A$ of $\mathsf{Suz}$.
Now there is a known typo
in \cite[p.131]{Atlas}, the shape of the maximal subgroup
in question being erroneously stated
as $3_2.U_4(3).2_3$. This has been corrected in the reprint of 2003, 
but can also be explicitly checked using {\sf GAP} \cite{GAP}:
the character tables of $\mathsf{Suz}$,
see \cite[p.131]{Atlas}, and the various bicyclic extensions of $U_4(3)$,
see \cite[pp.52--59]{Atlas}, show that the shape is as asserted above.
This implies (ii).

Finally, to show (iii), we use {\sf GAP} \cite{GAP} and
the explicit choices of subgroups made in the 
proof of {\bf\ref{FindH}} and in {\bf\ref{NotationH}},
to determine $N_G(R)$ explicitly. It turns out that
$N_G(R)$ is a subgroup of $H$ of index $2$.
Hence we have $N_G(R)=N_H(R)$, and as above we conclude that 
$N_G(R)$ is of shape $P.U_4(3).2^2$
and hence $G_R$ is of shape $P.U_4(3).2$;
but note that $N_H(R)\neq N_H(Q)$.
Computing $G_R$ and the character table of $G_R/R$ 
explicitly, and comparing with those of the various bicyclic extensions 
of $U_4(3)$,
we conclude that $G_R/R$ is of the shape asserted, thus (iii) follows.
\end{proof}

\begin{Lemma}\label{Suz2}
Set $\widetilde{G} = \mathsf{Suz}$.
\begin{enumerate}\renewcommand{\labelenumi}{\rm{(\roman{enumi})}}
    \item 
There is a unique block algebra $\widetilde{A}$ of $k\widetilde{G}$ 
with a defect group $D\cong C_3$; the non-trivial elements of $D$
belong the conjugacy class $3A$ of $\widetilde{G}$. 
    \item
We can write 
${\mathrm{Irr}}(\widetilde{A}) = \{ \chi_{16}, \chi_{38}, \chi_{41} \}$
such that
$\chi_{16}(1) = 18954$, $\chi_{38}(1) = 189540$, $\chi_{41}(1) = 208494$
and 
$\chi_{16}(u) = \chi_{38}(u) = 729$
for any element $u$ belonging to the conjugacy class ${\mathrm{3A}}$.
Moreover, we can write
${\mathrm{IBr}}(\widetilde{A}) = \{ \varphi_1, \varphi_2 \}$
such that the $3$-decomposition matrix of $\widetilde{A}$ is as follows:
\begin{center}
{
{\rm
\begin{tabular}{l|cc}
            & $\varphi_1$ & $\varphi_2$   \\
\hline
$\chi_{16}$ & 1  &  . \\
$\chi_{38}$ & .  &  1 \\
$\chi_{41}$ & 1  &  1 \\
\end{tabular} 
}}
\end{center}
Further, the simple $k\widetilde{G}$-modules in $\widetilde{A}$
affording $\varphi_1$ and $\varphi_2$ 
are trivial-source $k\widetilde{G}$-modules.
     \item
Set $\widetilde{H} = N_{\widetilde{G}}(D)$. Then 
$\widetilde{H} = 3_2.U_4(3).2_3'$.
      \item
Let $\widetilde{B}$ be the block algebra of $k\widetilde{H}$
that is the Brauer correspondent of $\widetilde{A}$.
Let further $\widetilde{\mathfrak f}$ be the Green
correspondence with respect to 
$(\widetilde{G} \times \widetilde{G}, \Delta D, 
  \widetilde{G} \times \widetilde{H})$.
Then, $\widetilde{\mathfrak f}$ induces a Puig equivalence between 
$\widetilde{A}$ and $\widetilde{B}$.
\end{enumerate}
\end{Lemma}

\begin{proof}
(i)-(ii) follow from calculations by {\sf GAP} \cite{GAP},
using the ordinary and Brauer character tables of $\mathsf{Suz}$, see
see \cite[p.128ff.]{Atlas} and \cite[$\mathsf{Suz}$ (mod 3)]{ModularAtlas}.

(iii) follows from \cite[p.131]{Atlas}, as was already remarked earlier
in the proof of {\bf\ref{CenQ}}.

(iv) follows from (i)-(iii) and 
\cite[Theorem 1.2]{KoshitaniKunugi2010}.
\end{proof}

\begin{Lemma}\label{LocalPuig}
Let $\mathcal M_Q$ be the unique (up to isomorphism)
indecomposable direct summand of
$A_Q{\downarrow}^{G_Q \times G_Q}_{G_Q \times H_Q}{\cdot}1_{B_Q}$
with vertex $\Delta P$.
Then, the pair $(\mathcal M_Q, \mathcal M_Q^{\vee})$ induces
a Puig equivalence between $A_Q$ and $B_Q$.
\end{Lemma}

\begin{proof}
Note first that $\mathcal M_Q$ exists by
\cite[2.4.Lemma]{KoshitaniKunugiWaki2008},
and also that $P$ is a defect group of $A_Q$ and $B_Q$
by \cite[7.6]{Linckelmann2001}. Now we follow the strategy
already employed in \cite[Proof of 6.2.Lemma]{KoshitaniKunugiWaki2008}:
Using {\bf\ref{CenQ}}(i) and (ii), as well as {\bf\ref{Suz2}}(iv),
the assertion follows by going over to the central quotients 
$G_Q/Q$ and $H_Q/Q$ and their blocks $\widetilde A$ and $\widetilde B$
dominating $A_Q$ and $B_Q$, respectively, 
\cite[Theorem]{KoshitaniKunugi2005}.
\end{proof}

\begin{Lemma}\label{StableEquivalenceAandB}
Let $\mathcal M$ be the unique (up to isomorphism)
indecomposable direct summand of the $(A,B)$-bimodule
$A{\downarrow}^{G \times G}_{G \times H}{\cdot}1_{B}$
with vertex $\Delta P$. Then the pair 
$(\mathcal M, \mathcal M^{\vee})$ induces a
splendid stable equivalence of Morita type 
between the block algebras $A$ and $B$.
\end{Lemma}

\begin{proof}
Note first that $\mathcal M$ exists, again, by
\cite[2.4.Lemma]{KoshitaniKunugiWaki2008}.
Then, by \cite[Theorem]{KoshitaniLinckelmann2005},
we have $\mathcal M_Q\cong e_Q\cdot{\mathcal M}(\Delta Q)\cdot f_Q$
as $(A_Q,B_Q)$-bimodules. 
Now by {\bf\ref{LocalPuig}}, and the fact from {\bf\ref{CenQ}}(iv) 
that $G_R = H_R$, the gluing theorem 
\cite[3.1.Theorem]{Linckelmann2001} implies the assertion; 
note that the fusion condition in the latter theorem
is automatically satisfied by \cite[1.15.Lemma]{KoshitaniKunugiWaki2004}.
\end{proof}

\begin{Lemma}\label{StableEquivalenceAandA'}
The block algebras $A$ and $A'$ are 
are splendidly stably equivalent of Morita type.
\end{Lemma}

\begin{proof}
This follows immediately by {\bf\ref{StableEquivalenceAandB}},
{\bf\ref{StructureOfB}}(iii) and
{\bf\ref{StableEquivalenceB'andA'}}.
\end{proof}

\section{Image of the stable equivalence}\label{image}

\begin{Notation}
Recall the notation in {\bf\ref{NotationCo1}}, {\bf\ref{NotationGprime}}, 
{\bf\ref{NotationHprime}}, and {\bf\ref{NotationH}}.
We denote by $f := f_{(G,P,H)}$ the Green correspondence
with respect to $(G,P,H)$.
Moreover, let $\IBr(A) := \{ S_1, \ldots, S_5 \}$  
and $\alpha$ be as in {\bf\ref{DecompositionMatrixOfA}}.
Finally, let $\mathcal M$ and $\mathcal M'$ be the
bimodules from {\bf\ref{StableEquivalenceAandB}}
and {\bf\ref{StableEquivalenceB'andA'}}, respectively.
\end{Notation}

\begin{Strategy}\label{strategy}
By {\bf\ref{StructureOfB}}(iii) there is a $(B,B')$-bimodule
$\mathfrak M$ inducing a Puig equivalence between $B$ and $B'$.
But there is a little bit of freedom in chosing $\mathfrak M$,
which we are going to exploit. 
Anyway, for any admissible 
choice of $\mathfrak M$, a splendid stable equivalence between 
$A$ and $A'$, as in {\bf\ref{StableEquivalenceAandA'}},
is afforded by the $(A, A')$-bimodule
$\mathcal M\otimes_{B} \mathfrak M \otimes_{B'} {\mathcal M'}^{\vee}$.
This yields the functor
$$ F: \mathrm{mod}{\text{-}}A\rightarrow \mathrm{mod}{\text{-}}A' :
X \mapsto X \otimes_A \mathcal M\otimes_{B} \mathfrak M \otimes_{B'} 
                    {\mathcal M'}^{\vee} ,$$
which induces an equivalence
$\underline{\mathrm{mod}}{\text{-}}A\rightarrow
 \underline{\mathrm{mod}}{\text{-}}A'$
between the respective stable module categories.
Recall that hence, by \cite[Theorem 2.1(ii)]{Linckelmann1996MathZ},
the functor $F$ maps each simple $kG$-module in $A$ to an
indecomposable $kG'$-module in $A'$. 

Our practical aim in this section now is to prove that
$\mathfrak M$ can be chosen such that $F$ actually transfers 
each simple $kG$-module in $A$ to a simple $kG'$-module. 
Our choice of $\mathfrak M$ is specified below in {\bf\ref{S1}}.
If this has been achieved, then Linckelmann's Theorem
\cite[Theorem 2.1(iii)]{Linckelmann1996MathZ}  
yields that $F$ realizes a Morita equivalence between $A$ and $A'$,
and thus even a Puig equivalence between these block algebras.
This will then also decide the missing entries in the
$3$-decomposition matrix of $A$, see {\bf\ref{DecompositionMatrixOfA}}.
\end{Strategy}

\begin{Lemma}\label{TrivialSourceModulesInA}
The following characters are afforded by direct sums of 
of trivial-source $A$-modules:
\begin{enumerate} \renewcommand{\labelenumi}{\rm{(\roman{enumi})}}
    \item 
$\chi_{29}$
    \item
$\chi_{29} + \chi_{38}$
    \item
$\chi_{38} + \chi_{62}$
    \item
$\chi_{29} + \chi_{38} + \chi_{55}$
    \item
$\chi_{29} + \chi_{62} + \chi_{89}$
\end{enumerate}
\end{Lemma}

\begin{proof} 
Using {\sf GAP} \cite{GAP}, and the character tables of $G$ and its 
maximal subgroups, inducing suitable linear characters we determine 
various permutation characters of $G$, 
and their components belonging to the block $A$. Doing so we find
$$ (1_{2^{1+8}_+.O_8^+(2)})\uparrow^G\cdot 1_A =\chi_{29}
\quad\text{and}\quad
(1_{U_6(2).\mathfrak S_3})\uparrow^G\cdot 1_A =\chi_{29}+\chi_{38} ,$$
where $1$ denotes the respective trivial character,
verifying (i) and (ii). Moreover, (iii) follows from
$$ (\lambda_{2^{4+12}.(\mathfrak S_3\times 3\mathfrak S_6)})
   \uparrow^G\cdot 1_A =\chi_{38}+\chi_{62} ,$$
where $\lambda$ is the inflation to 
$2^{4+12}.(\mathfrak S_3\times 3\mathfrak S_6)$ of the unique 
linear character of $\mathfrak S_3\times 3\mathfrak S_6$
having kernel $3\times 3\mathfrak S_6$.
Finally, for (iv) and (v) we observe
$$ (1^-_{2^{2+12}:(\mathfrak A_8\times \mathfrak S_3)})
   \uparrow^G\cdot 1_A =\chi_{29}+\chi_{38}+\chi_{55}
\quad\text{and}\quad
   (1_{2^{2+12}:(\mathfrak A_8\times \mathfrak S_3)})
   \uparrow^G\cdot 1_A =\chi_{29}+\chi_{62}+\chi_{89} ,$$ 
where $1^-$ denotes the unique non-trivial linear character of
$2^{2+12}:(\mathfrak A_8\times \mathfrak S_3)$.
\end{proof}

\begin{Lemma}\label{CharacterValues}
The characters in $A$ have the following values at elements of order $3$:
\begin{center}
{
{\rm
\begin{tabular}{l|r|r}
            & $3A$ & $3B$   \\
\hline
$\chi_{29}$ & $18954$ & $729$ \\
$\chi_{38}$ & $-18954$& $1458$\\
$\chi_{51}$ & $189540$& $729$ \\
$\chi_{55}$ & $208494$& $-729$\\
$\chi_{62}$ & $18954$ & $729$ \\
$\chi_{80}$ & $208494$& $-729$\\
$\chi_{85}$ & $-189540$&$1458$\\
$\chi_{89}$ & $189540$& $729$ \\
$\chi_{91}$ & $-208494$&$-1458$
\end{tabular} 
}}
\end{center}
\end{Lemma}

\begin{proof}
This follows from \cite[pp.184--186]{Atlas}.
\end{proof}

\begin{Lemma}\label{Green}
The following holds:
\begin{enumerate}
  \renewcommand{\labelenumi}{\rm{(\roman{enumi})}}
    \item
$S_1$ is a trivial-source $kG$-module with vertex $P$ and 
associated character $\chi_{29}$.
    \item
Using the notation in {\bf\ref{StructureOfB}}(ii) we have
$f(S_1) \in \{ 729a, 729b, 729c, 729d \}$.
\end{enumerate}
\end{Lemma}

\begin{proof} 
(i) From {\bf\ref{TrivialSourceModulesInA}}(i), we know that
$S_1$ is a trivial-source module with associated character $\chi_{29}$,
and, by Kn\"orr's result \cite[3.7.Corollary]{Knoerr}, 
$S_1$ has $P$ as its vertex.

(ii) 
Setting $T_1 := f(S_1)$, by (i) and \cite[Lemma 2.2]{Okuyama1981} 
we have $T_1 \in \IBr(B)$. Set $\dim_k(T_1) = 729\cdot d$, 
where $d\in\{1,2\}$.
By the definition of Green correspondence, we have
${T_1}{\uparrow}^G = S_1 \oplus X$, for a $C_3$-projective $kG$-module $X$.
Then we have
$3^{9-1}|\dim_k(X)$ by \cite[Theorem 4.7.5]{NagaoTsushima},
implying that $\dim_k({T_1}{\uparrow}^G) \equiv \dim_k(S_1)$ (mod $3^8$).
Now by (i) we have 
$\dim_k(S_1) = \chi_{29}(1) = 2816586 \equiv 2187$ (mod $3^8$),
and
$\dim_k({T_1}{\uparrow}^G) = |G:H|{\cdot}\dim_k (T_1)
 = 17681664000 \cdot 729 \cdot d \equiv 2187 \cdot d$ (mod $3^8$),
implying that $d = 1$.
\end{proof}

\begin{Lemma}\label{S1}
The following holds:
\begin{enumerate}
  \renewcommand{\labelenumi}{\rm{(\roman{enumi})}}
    \item
There is an $(B,B')$-bimodule $\mathfrak M$ inducing a 
Puig equivalence between $B$ and $B'$, such that
$S_1 \otimes_A \mathcal M \otimes_B \mathfrak M = 1a$.
    \item
Choosing $\mathfrak M$ like this we have $F(S_1)=1a$.
\end{enumerate}
\end{Lemma}

\begin{proof}
(i)
By {\bf\ref{StructureOfB}}(iii) there is a $(B,B')$-bimodule
$\mathfrak M$ inducing a Puig equivalence between $B$ and $B'$.
Hence the functor $-\otimes_{B}\mathfrak M$ induces a bijection
from $\{ 729a, 729b, 729c, 729d \}$ to $\{1a,1b,1c,1d\}$.
Thus by {\bf\ref{Green}}(ii), 
using \cite[Lemma A.3]{KoshitaniMuellerNoeske2011}, we have
$S_1\otimes_A \mathcal M \otimes_B \mathfrak M = 
f(S_1)\otimes_B \mathfrak M = 1x$,
for some $x\in\{a,b,c,d\}$.
Now tensoring with $1x$ induces a Puig auto-equivalence of $B'$, see
\cite[Lemma 2.8]{KoshitaniKunugiWaki2008}.
Thus by replacing $\mathfrak M$ by $\mathfrak M\otimes 1x$,
the assertion follows from observing that $1x\otimes 1x=1a$
as $kH'$-modules.

(ii) follows from 
$F(S_1)=1a \otimes_{B'} {\mathcal M'}^{\vee} = 1a$,
using \cite[Lemma A.3]{KoshitaniMuellerNoeske2011},
and the fact that Green correspondence maps the trivial
module to the trivial module.
\end{proof}

\begin{Notation}\label{choice}
>From now on let $\mathfrak M$ be chosen as in {\bf\ref{S1}},
and let $F$ be the associated functor as described in {\bf\ref{strategy}}.
\end{Notation}

\begin{Lemma}\label{S2}
The following holds:
\begin{enumerate}
  \renewcommand{\labelenumi}{\rm{(\roman{enumi})}}
    \item
There is a trivial-source $kG$-module $T$ in $A$ with
associated character $\chi_{29}+\chi_{38}$ and vertex $R$, so that
$T$ is uniserial with Loewy and socle series
$$ T=\boxed{\begin{matrix} S_1 \\ S_2 \\ S_1 \end{matrix}} .$$
    \item
We have
\[  F(T) = 
    \boxed{\begin{matrix} 1a \\ 4b \\ 1a \end{matrix}}
    \oplus \text{\rm{(proj)}}
\quad\text{and}\quad
F(S_2) = 4b .\]
\end{enumerate}
\end{Lemma}

\begin{proof} 
(i)
>From {\bf\ref{TrivialSourceModulesInA}}(ii)
we know that there is a self-dual $kG$-module $T$ in $A$
which is a direct sum of trivial-source modules and has
character $\chi_{29}+\chi_{38}$. Thus, by {\bf\ref{DecompositionMatrixOfA}},
$T=2\cdot S_1 + S_2$ as composition factors, and hence from
$[T,T]^G=2$ we conclude that $T$ is indecomposable, and thus
of shape as asserted. 
Finally, $R$ is a vertex of $T$ by making use of  
{\bf\ref{CharacterValues}} and \cite[II Lemma 12.6(ii)]{Landrock}.

(ii)
be the projective-free part of $F(T)$.
Then, by {\bf\ref{S1}}(ii),
\cite[Lemma A.3]{KoshitaniMuellerNoeske2011}, and 
{\bf\ref{TrivialSourceModulesInAprime}}(iii) we conclude that
$X$ is of the shape asserted.
Now the splitting-off method, 
see \cite[Lemma A.1]{KoshitaniMuellerNoeske2011},
yields $F(S_2)=4b$. 
\end{proof}

\begin{Lemma}\label{S3}
The following holds:
\begin{enumerate}
  \renewcommand{\labelenumi}{\rm{(\roman{enumi})}}
    \item
There is a trivial-source $kG$-module $U$ in $A$ with
associated character $\chi_{38}+\chi_{55}$ and vertex $P$, so that
$U$ has Loewy and socle series
$$ U = \boxed{ \begin{matrix} S_2 \\ S_1 \ \ S_3 \\ S_2 \end{matrix} } .$$ 
    \item
We have
$$ F(U) = \boxed{ \begin{matrix} 4b \\ 1a \ \ 1b \\ 4b \end{matrix} }
  \oplus \text{\rm{(proj)}} 
\quad\text{and}\quad
F(S_3)=1b .$$
    \item
$S_3$ is a trivial-source $kG$-module in $A$ with associated character 
$\chi_{51}$ and vertex $P$.
\end{enumerate}
\end{Lemma}  

\begin{proof}
(i)
By {\bf\ref{TrivialSourceModulesInA}}(iv), 
there is a self-dual $kG$-module $X$ being the direct sum 
of trivial-source $kG$-modules and having character
$\chi_{29}+\chi_{38}+\chi_{55}$.
Hence, by {\bf\ref{DecompositionMatrixOfA}}, we have 
$X = 2\cdot S_1+ 2\cdot S_2 + S_3$ as composition factors.

Assume first that $X$ is indecomposable. Then from self-duality and
$[U,U]^G=3$ we conclude that
$U$ has Loewy and socle series
$$ U=\boxed{\begin{matrix} S_1 \ \ S_2 \\ S_3 \\ S_1 \ \ S_2 \end{matrix}} .$$
But from $\gen{\chi_{29}+\chi_{38},\chi_{29}+\chi_{38}+\chi_{55}}_G=2$
and the structure of $T$ in {\bf\ref{S2}}(i) we conclude that
there is an embedding of $T$ into $X$, and hence $X$ has 
$T\oplus S_2$ as submodule, implying that $S_3$ is an epimorphic
image, and thus a direct summand of $X$, a contradiction.

Hence $X$ is decomposable. Assume now that there is a direct
summand with character $\chi_{38}$, or a direct
summand with character $\chi_{55}$.
Then in either case it follows from self-duality that $X$ 
has a direct summand isomorphic to $S_2$. 
But by {\bf\ref{DecompositionMatrixOfA}} the module $S_2$ is not liftable,
hence is not a trivial-source module, a contradiction.

Thus we conclude that there is a trivial-source $kG$-module $U$ with 
character $\chi_{38}+\chi_{55}$, and again by self-duality $U$ has
shape asserted. 
Moreover, by {\bf\ref{CharacterValues}} and
\cite[II, Lemma 12.6(ii)]{Landrock},
$U$ has $P$ as its vertex. 

(ii)
Let $X$ be the projective-free part of $F(U)$.
Then, by {\bf\ref{S2}}(ii),
\cite[A.3 Lemma]{KoshitaniMuellerNoeske2011}, 
\cite[A.1 Lemma]{KoshitaniMuellerNoeske2011}, and 
{\bf\ref{TrivialSourceModulesInAprime}}(i), we get
that $X$ is a trivial-source $kG'$-module in $A'$ with vertex $P$ 
such that $[4b,X]^{G'} \neq 0$.
Hence, {\bf\ref{TrivialSourceModulesInAprime}}(i) yields 
the shape of $X$ as asserted. This, by {\bf\ref{S1}}(ii) 
and using stripping-off again, implies the statement on $F(S_3)$.

(iii)
follows from (ii), {\bf\ref{TrivialSourceModulesInAprime}}(i),
\cite[A.3 Lemma]{KoshitaniMuellerNoeske2011},
and {\bf\ref{DecompositionMatrixOfA}}; the
statement on the vertex also follows from 
Kn\"orr's result \cite[3.7.Corollary]{Knoerr}.
\end{proof} 

\begin{Lemma}\label{S4}
The following holds:
\begin{enumerate}
  \renewcommand{\labelenumi}{\rm{(\roman{enumi})}}
 \item
There is a trivial source $kG$-module $V$ in $A$ 
with associated character $\chi_{38}+\chi_{62}$ and vertex $R$,
so that $V$ has Loewy and socle series
$$ V = \boxed{ \begin{matrix} S_2 \\ S_1 \ \ S_4 \\ S_2 \end{matrix} } .$$
 \item
We have
$$ F(V) = \boxed{ \begin{matrix} 4b \\ 1a \ \ 6 \\ 4b \end{matrix} }
               \oplus \text{\rm{(proj)}}
\quad\text{and}\quad
F(S_4) = 6 .$$
\end{enumerate} 
\end{Lemma}

\begin{proof}
(i)
By {\bf\ref{TrivialSourceModulesInA}}(iii), 
there is a self-dual $kG$-module $V$ which is the direct sum 
of trivial-source $kG$-modules and has character $\chi_{38}+\chi_{62}$.
Hence, by {\bf\ref{DecompositionMatrixOfA}}, we have 
$V = S_1+ 2\cdot S_2 + S_4$ as composition factors.
Thus from $[V,V]^G=2$ we conclude that $V$ is indecomposable, and 
of shape as asserted.
Finally, $R$ is a vertex of $T$ by making use of
{\bf\ref{CharacterValues}} and \cite[II Lemma 12.6(ii)]{Landrock}.

(ii)
Let $X$ be the projective-free part of $F(V)$.
Then, by {\bf\ref{S2}}(ii),
\cite[A.3 Lemma]{KoshitaniMuellerNoeske2011}, 
\cite[A.1 Lemma]{KoshitaniMuellerNoeske2011}, and 
{\bf\ref{TrivialSourceModulesInAprime}}(iii), we get
that $X$ is a trivial-source $kG'$-module in $A'$ with vertex $R$ 
such that $[4b,X]^{G'} \neq 0$. 
Hence, {\bf\ref{TrivialSourceModulesInAprime}}(iii) yields 
the shape of $X$ as asserted. This, by {\bf\ref{S1}}(ii) 
and using stripping-off again, implies the statement on $F(S_4)$.
\end{proof}

\begin{Lemma}\label{ext}
We have $\Ext^1_{kG}(S_4, S_3) = 0 = \Ext^1_{kG}(S_3, S_4)$.
\end{Lemma}

\begin{proof}
{\bf\ref{S4}}(ii),{\bf\ref{S3}}(ii), and 
{\bf\ref{TrivialSourceModulesInAprime}}(iv) that
$\Ext^1_A(S_4,S_3) \cong \Ext^1_{A'}(F(S_4),F(S_3))\cong 
 \Ext^1_{A'}(6,1b) = 0$.
The second statement follows similarly.
\end{proof}

\begin{Lemma}\label{S5}
The following holds:
\begin{enumerate}
  \renewcommand{\labelenumi}{\rm{(\roman{enumi})}}
 \item
There is a trivial source $kG$-module $W$ in $A$ 
with associated character $\chi_{62}+\chi_{89}$ 
and vertex $P$, such that $W$ has Loewy and socle series
$$ W = \boxed{ \begin{matrix} S_4 \\ S_2 \ \ S_5 \\ S_4 \end{matrix} } .$$
 \item
We have $\alpha=0$.
 \item
We have
$$ F(W) = \boxed{ \begin{matrix} 6 \\ 4a \ \ 4b \\ 6 \end{matrix} }
               \oplus \text{\rm{(proj)}}
\quad\text{and}\quad
F(S_5) =  4a .$$
\end{enumerate}
\end{Lemma}

\begin{proof}
(i)--(ii)
By {\bf\ref{TrivialSourceModulesInA}}(v), 
there is a self-dual $kG$-module $X$ being the direct sum 
of trivial-source $kG$-modules and having character
$\chi_{29}+\chi_{62}+\chi_{89}$.
Hence, by {\bf\ref{DecompositionMatrixOfA}}, we have 
$X = S_1 + S_2 + \alpha\cdot S_3 + 2\cdot S_4 + S_5$ as composition factors.
Now, from $\gen{\chi_{29},\chi_{29}+\chi_{62}+\chi_{89}}_G=1$,
using {\bf\ref{Green}}(i), we infer that $[S_1,X]^G=1=[X,S_1]^G$, 
thus $S_1$ is a direct summand of $X$, so that there is a trivial-source
module $W$ with character $\chi_{62}+\chi_{89}$, and thus
$W = S_2 + \alpha\cdot S_3 + 2\cdot S_4 + S_5$ as composition factors.

Next, from $\gen{\chi_{51},\chi_{62}+\chi_{89}}_G=0$ 
and {\bf\ref{S3}}(iii) we conclude that 
$[S_3,W]^G=0=[W,S_3]^G$, that is $S_3$ does not occur
neither in the socle nor the head of $W$.
Moreover, since both $S_2$ and $S_5$ are not liftable,
by {\bf\ref{DecompositionMatrixOfA}}, neither of them is a
trivial-source module, hence we infer that 
$[S_2,W]^G = 0 = [W,S_2]^G$ and $[S_5,W]^G = 0 = [W,S_5]^G$.
Thus we conclude that 
$W/{\mathrm{rad}}(W) \cong {\mathrm{soc}}(W) \cong S_4$,
in particular $W$ is indecomposable.
By making use of
{\bf\ref{CharacterValues}} and \cite[II Lemma 12.6(ii)]{Landrock}
we get that $P$ is a vertex of $W$.

Thus for the heart $Y:={\mathrm{rad}}(W)/{\mathrm{soc}}(W)$ of $W$
we have $Y = S_2+S_5+\alpha\cdot S_3$ as composition factors. 
Now assume that $\alpha\neq 0$. Then, by self-duality, we infer that
$[S_3,Y]^G=1=[Y,S_3]^G$. Hence, since 
$W/{\mathrm{rad}}(W) \cong {\mathrm{soc}}(W) \cong S_4$,
we infer that there are uniserial modules of shape
$$ \boxed{ \begin{matrix} S_3 \\ S_4 \end{matrix} } 
\quad\text{and}\quad
   \boxed{ \begin{matrix} S_4 \\ S_3 \end{matrix} } 
$$
contradicting {\bf\ref{ext}}.
Hence we infer $\alpha=0$, and $W$ has shape as asserted.

(iii)
Let $X$ be the projective-free part of $F(W)$.
Then, by {\bf\ref{S4}}(ii),
\cite[A.3 Lemma]{KoshitaniMuellerNoeske2011}, 
\cite[A.1 Lemma]{KoshitaniMuellerNoeske2011}, and 
{\bf\ref{TrivialSourceModulesInAprime}}(i), we get
that $X$ is a trivial-source $kG'$-module in $A'$ with vertex $P$ 
such that $[6,X]^{G'} \neq 0$. 
Hence, {\bf\ref{TrivialSourceModulesInAprime}}(i) yields 
the shape of $X$ as asserted. This, by {\bf\ref{S2}}(ii) 
and using stripping-off again, implies the statement on $F(S_5)$.
\end{proof}

\renewcommand{\proofname}{\bf Proof of Theorem {\ref{MoritaEquivalentToS6}}}
\begin{proof}
First, $A$ and $A'$ are splendidly stably equivalent of Morita type
via the functor $F$ by {\bf\ref{StableEquivalenceAandA'}} 
and {\bf\ref{strategy}}. Moreover, by the choice in 
{\bf\ref{choice}}, as well as
{\bf\ref{S1}}(ii), {\bf\ref{S2}}(ii), {\bf\ref{S3}}(ii),
{\bf\ref{S4}}(ii), and {\bf\ref{S5}}(iii), all simple $A$-modules are sent
to simple $A'$-modules via the functor $F$.
Hence, \cite[Theorem 2.1(iii)]{Linckelmann1996MathZ} yields
that $A$ and $A'$ are Morita equivalent, and hence are Puig equivalent.
\end{proof}

\renewcommand{\proofname}{\bf Proof of Theorem {\ref{MainTheorem}}}
\begin{proof}
Just as in \cite[(6.13)]{KoshitaniKunugiWaki2008}, we conclude that
$A$ and $B$ are splendidly derived equivalent, by using
the splendid derived equivalence between $A'$ and $B'$ given
in \cite[Example 4.4]{Okuyama1997} and \cite[Corollary 2]{Okuyama2000}.
\end{proof}

\renewcommand{\proofname}
             {\bf Proof of Corollary {\ref{ADGCforCo1forAllBlocks}}}
\begin{proof}
Using the character table of $G$, {\sf GAP} \cite{GAP} shows
that $G$ has five $3$-blocks, of defect $9$, $3$, $2$, $1$, and $0$.
While the former two, according to \cite{Noeske2008},
have non-abelian defect groups, the latter two have
cyclic and trivial defect groups, respectively,
for which both Conjectures {\bf\ref{ADGC}} and {\bf\ref{RickardConjecture}}
are well-known to hold. Hence in this respect the only block of 
interest is the one of defect $2$ under consideration.
\end{proof}

\renewcommand{\proofname}{\bf Proof of {\ref{OkuyamaWaki}}}
\begin{proof}
This follows from {\bf\ref{MoritaEquivalentToS6}} and
\cite[Theorem 2.3]{OkuyamaWaki1998}, see 
\cite[Remark 3.7]{Okuyama1997}.
\end{proof}
\renewcommand{\proofname}{\it Proof}

\bigskip

{\begin{center}
{\small\sc Acknowledgements}
\end{center}}

\small\noindent {\rm
This work was done while the first author was
staying in RWTH Aachen University in 2011 and 2012. 
He is grateful to Gerhard Hiss for his kind hospitality.
For this research the first author was partially
supported by the Japan Society for Promotion of Science (JSPS),
Grant-in-Aid for Scientific Research (C)23540007, 2011--2014. 

The second author is grateful for financial support in
the framework of the DFG (German Science Foundation) Scientific Priority
Programmes SPP-1388 `Representation Theory' and SPP-1489 `Computer Algebra',
to which this research is a contribution.}

\end{document}